\documentclass{amsart}
\usepackage{amsmath}
\usepackage{amscd}
\usepackage{amssymb}
\usepackage{amsthm}
\RequirePackage{filecontents}
\usepackage{fullpage}
\usepackage[numbers]{natbib}  
\usepackage[draft]{hyperref}
\theoremstyle{plain}
\newtheorem{prop}{Proposition}
\newtheorem{thm}[prop]{Theorem}
\newtheorem{lem}[prop]{Lemma}

\numberwithin{prop}{section}

\theoremstyle{definition}

\newtheorem{rem}[prop]{Remark}
\newtheorem{ex}[prop]{Example}

\newenvironment{psmallmatrix}
  {\left(\begin{smallmatrix}}
  {\end{smallmatrix}\right)}

\author{Brandon Williams }

\address{Fachbereich Mathematik \\ Technische Universit\"at Darmstadt \\ 64289 Darmstadt, Germany}

\email{bwilliams@mathematik.tu-darmstadt.de}

\begin{document}

\nocite{*}

\title{A construction of antisymmetric modular forms for Weil representations}

\begin{abstract} We give coefficient formulas for antisymmetric vector-valued cusp forms with rational Fourier coefficients for the Weil representation associated to a finite quadratic module. The forms we construct always span all cusp forms in weight at least three. These formulas are useful for computing explicitly with theta lifts.
\end{abstract}

\maketitle

\section{Introduction}

This note is an extended version of chapter 7 of the author's dissertation and is in some sense a continuation of \cite{W}. Its purpose is to give formulas for a spanning set of vector-valued cusp forms with rational Fourier coefficients for the (dual) Weil representation $\rho^*$ attached to a finite quadratic module $(A,Q)$ which are antisymmetric under the action of $-I \in \mathrm{SL}_2(\mathbb{Z})$. Equivalently the weight $k$ of these cusp forms is such that $k + \mathrm{sig}(A,Q)/2$ is odd, where $\mathrm{sig}(A,Q)$ is the signature of $(A,Q)$. \\

Bases of modular forms with rational coefficients are known to exist due to the work of McGraw \cite{Mc}. On the other hand, all algorithms to compute such bases in the literature that the author is aware of (e.g. \cite{R}, \cite{W}) assume that $k + \mathrm{sig}(A,Q) / 2$ is even. Computing antisymmetric modular forms has received less attention; the first effective formula to compute the space of Eisenstein series in antisymmetric weights for arbitrary $(A,Q)$ was given in \cite{Sch}. The computation of cusp forms here complements this. The most important application of antisymmetric vector-valued modular forms is that they are mapped to orthogonal modular forms under the additive theta lift (of Gritsenko, Kudla, Oda, Rallis-Schiffmann and many others).

Our main results are the two theorems below. (The terms and notation are explained in section two.)

\begin{thm} Let $(A,Q)$ be a finite quadratic module, and let $k \ge 3$ be a weight for which $k + \mathrm{sig}(A,Q)/2$ is an odd integer. For any $\beta \in A$ and $m \in \mathbb{Z} - Q(\beta)$, $m > 0$, let $R_{k,m,\beta}$ be the cusp form defined through the Petersson scalar product by $$(f,R_{k,m,\beta}) = 2 \cdot \frac{\Gamma(k-1)}{(4\pi m)^{k-1}}L_{m,\beta}(f,2k-1) \; \text{for all cusp forms} \; f,$$ where $L_{m,\beta}$ is essentially a rescaled symmetric square $L$-function: $$L_{m,\beta}(f,s) = \sum_{\lambda =1}^{\infty} \frac{c(\lambda^2 m, \lambda \beta)}{\lambda^s} \; \; \text{if} \; \; f(\tau) = \sum_{\gamma \in A} \sum_{n \in \mathbb{Z} - Q(\gamma)} c(n,\gamma) q^n \mathfrak{e}_{\gamma}, \; q = e^{2\pi i \tau}.$$ Then all $R_{k,m,\beta}$ have rational Fourier coefficients, and there is a finite collection of indices $(m,\beta)$ for which the forms $R_{k,m,\beta}$ span the entire cusp space $S_k(\rho^*)$.
\end{thm}

\begin{thm} Let $(\Lambda,Q)$ be an even lattice which realizes the discriminant group $A = \Lambda'/\Lambda$, and let $k \ge 4$ be a weight for which $k + \mathrm{sig}(\Lambda)/2$ is odd. For any $\beta \in \Lambda'$ and $m \in \mathbb{Z} - Q(\beta)$, $m > 0$, let $\Lambda_{m,\beta}$ denote the even lattice with underlying group $\Lambda \oplus \mathbb{Z}$ and quadratic form $Q_{m,\beta}(v,\lambda) = Q(v + \lambda \beta) + m\lambda^2.$ Let $c_{m,\beta}(n,\gamma)$ denote the Fourier coefficients of the weight $k-3/2$ Eisenstein series for the dual Weil representation attached to $\Lambda_{m,\beta}$ (as in \cite{BK}), i.e. $$E_{k-3/2}(\tau;\Lambda_{m,\beta}) = \sum_{\gamma \in \Lambda_{m,\beta}' / \Lambda_{m,\beta}} \sum_{n \in \mathbb{Z} - Q_{m,\beta}(\gamma)} c_{m,\beta}(n,\gamma) q^n \mathfrak{e}_{\gamma}.$$ Then $R_{k,m,\beta}$ is given explicitly by the formula $$R_{k,m,\beta}(\tau) = \frac{1}{2m} \sum_{\gamma \in A} \sum_{n \in \mathbb{Z} - Q(\gamma)} \Big[ \sum_{r \in \mathbb{Z} - \langle \gamma, \beta \rangle} r \cdot c_{m,\beta}\Big(n, (\gamma - \frac{r}{2m}\beta, \frac{r}{2m} \beta) \Big) \Big] q^{n + r^2 / 4m} \mathfrak{e}_{\gamma}.$$ 
\end{thm}

This rest of this note is organized as follows. Sections 2 contains background material on vector-valued modular forms for Weil representations. Section 3 constructs the cusp forms $R_{k,m,\beta}$ and proves Theorem 1.2. In section 4 we adapt the construction to small weights and completes the proof of Theorem 1.1 and indicate how the formula in weight three can be used to compute identities among class numbers. Section 5 discusses the main application i.e. the theta lift and orthogonal modular forms. We compute the theta lift of $R_{k,m,\beta}$ and give examples for lattices of signature $(2,1)$ and $(2,2)$ where the lift can be interpreted as an elliptic modular form and a Hilbert modular form respectively. \\

An implementation of the coefficient formula (Theorem 1.2) and its adaptation to low weights in SAGE is available on the author's university webpage. \\

\textbf{Acknowledgments:} I am grateful to Richard Borcherds for supervising the dissertation this note is based on, and for many discussions when I was a graduate student to which I owe my interest in vector-valued modular forms. I also thank Jan Hendrik Bruinier and Martin Raum for helpful discussions, and I thank the reviewer for suggestions which improved the structure of this note. This work was supported by the LOEWE research unit Uniformized Structures in Arithmetic and Geometry.

\section{Finite quadratic modules and modular forms}

This section reviews finite quadratic modules, their Weil representations, and vector-valued modular forms for those representations in order to fix some conventions and notation for the rest of the note. For a more thorough introduction see chapter 14 of \cite{CS}. \\

A \textbf{finite quadratic module} $(A,Q)$ consists of a finite abelian group $A$ and a nondegenerate $(\mathbb{Q}/\mathbb{Z})$-valued quadratic form $Q$ on it. (In other words the bilinear form $\langle x,y \rangle = Q(x+y) - Q(x) - Q(y)$ is nondegenerate.) Given this data there is a unitary representation $\rho^* = \rho^*_{(A,Q)}$ of the metaplectic group $\tilde \Gamma = \mathrm{Mp}_2(\mathbb{Z})$ on the group ring $\mathbb{C}[A]$, through which the generators $S = (\begin{psmallmatrix} 0 & -1 \\ 1 & 0 \end{psmallmatrix},\sqrt{\tau})$ and $T = (\begin{psmallmatrix} 1 & 1 \\ 0 & 1 \end{psmallmatrix}, 1)$ act by $$\rho^*(S) \mathfrak{e}_{\gamma} = \frac{1}{\sqrt{|A|}} \mathbf{e}( \mathrm{sig}(A,Q) / 8) \sum_{\beta \in A} \mathbf{e}\Big( \langle \gamma, \beta \rangle \Big) \mathfrak{e}_{\beta}, \quad \; \; \rho^*(T) \mathfrak{e}_{\gamma} = \mathbf{e}(-Q(\gamma)) \mathfrak{e}_{\gamma}, \; \; \gamma \in A.$$ (Recall that elements of $\mathrm{Mp}_2(\mathbb{Z})$ may be understood as pairs $(M,\phi)$ where $M = \begin{psmallmatrix} a & b \\ c & d \end{psmallmatrix} \in \mathrm{SL}_2(\mathbb{Z})$ and where $\phi$ is a branch of $\sqrt{c \tau + d}$ on $\mathbb{H}$.) Here we are using the notation $\mathbf{e}(x) = e^{2\pi i x}$, and $\mathfrak{e}_{\gamma}$, $\gamma \in A$ denotes the canonical basis of $\mathbb{C}[A]$. In the most common convention $\rho^*$ is called the \textbf{dual Weil representation} associated to $(A,Q)$.  In the definition above, $\mathrm{sig}(A,Q) \in \mathbb{Z}/8\mathbb{Z}$ is the \text{signature} of $(A,Q)$, i.e. the signature mod $8$ of any even lattice whose discriminant is isomorphic to $(A,Q)$. The signature can be computed intrinsically by means of Milgram's formula (\cite{MH}, appendix 4): $$\mathbf{e}(\mathrm{sig}(A,Q)/8) = \frac{1}{\sqrt{|A|}} \sum_{\beta \in A} \mathbf{e}\Big( Q(\beta) \Big).$$

Let $\mathbb{H} = \{\tau = x+iy: \, y > 0\}$ be the upper half-plane. \textbf{Modular forms} for $\rho^*$ of weight $k \in \frac{1}{2}\mathbb{Z}$ are holomorphic functions $f = f(\tau) : \mathbb{H} \rightarrow \mathbb{C}[A]$ which remain bounded as $y = \mathrm{im}(\tau)$ tends to $\infty$ and which satisfy $$ f\left(\frac{a \tau + b}{c \tau + d}\right) = (c \tau + d)^k \rho^*\left( \begin{psmallmatrix} a & b \\ c & d \end{psmallmatrix}, \sqrt{c \tau + d} \right) f(\tau)$$ for all $M = 
\left( \begin{psmallmatrix} a & b \\ c & d \end{psmallmatrix}, \sqrt{c \tau + d} \right) \in \tilde \Gamma$. The space of modular forms of weight $k$ for $\rho^*$ will be denoted $M_k(\rho^*)$. \\

The transformation under $T$ implies that a modular form $f$ for $\rho^*$ has a Fourier expansion of the form $$f(\tau) = \sum_{\gamma \in A} \sum_{n \in \mathbb{Z} - Q(\gamma)} c(n,\gamma) q^n \mathfrak{e}_{\gamma}, \; \; c(n,\gamma) \in \mathbb{C}.$$ The element $Z = (-I,i) = S^2 \in \mathrm{Mp}_2(\mathbb{Z})$ acts through $\rho^*$ by $\rho^*(Z) \mathfrak{e}_{\gamma} = (-1)^{\mathrm{sig}(A,Q) / 2} \mathfrak{e}_{-\gamma}$ and it acts trivially on $\mathbb{H}$, so the transformation under $Z$ implies that $M_k(\rho^*) = 0$ if $k + \mathrm{sig}(A,Q)/2$ is not integral, and that the Fourier coefficients $c(n,\gamma)$ of any modular form $f(\tau)$ satisfy $$c(n,\gamma) = (-1)^{k+ \mathrm{sig}(A,Q) / 2} c(n,-\gamma).$$ Therefore it seems reasonable to refer to $k$ as a \textbf{symmetric} or \textbf{antisymmetric weight} when $k + \mathrm{sig}(A,Q)/2$ is respectively even or odd. \\

The spaces $M_k(\rho^*)$ of modular forms of weight $k$ are always finite-dimensional and their dimensions can be calculated with some effort with the Riemann-Roch theorem. An effective formula (in terms of certain Gauss sums, which can be computed with only one iteration over $A$) that is valid for both symmetric and antisymmetric weights appears as Theorem 2.1 of \cite{DHS}. \\

The easiest way to produce modular forms is by averaging. Let $\tilde \Gamma_{\infty} = \langle T, Z \rangle \le \tilde \Gamma$ be the stabilizer of the constant function $\mathfrak{e}_0$. For any $k > 2$ and any smooth function $\phi : \mathbb{H} \rightarrow \mathbb{C}[A]$ which satisfies $\phi(\tau) = (-1)^k \rho^*(Z) \phi(\tau)$ and $\phi(\tau + 1) = \rho^*(T) \phi(\tau)$, the \textbf{Poincar\'e series}, if it converges locally uniformly, is the series $$\mathbb{P}_k(\phi) = \sum_{M \in \tilde \Gamma_{\infty} \backslash \tilde \Gamma} \phi |_{k,\rho^*} M = \frac{1}{2} \sum_{\substack{c,d \in \mathbb{Z} \\ \mathrm{gcd}(c,d) = 1}} (c \tau + d)^{-k} \rho^*\left( \begin{psmallmatrix} a & b \\ c & d \end{psmallmatrix}, \sqrt{c \tau + d} \right)^{-1} \phi\left( \frac{a \tau + b}{c \tau + d} \right).$$

In the sum on the right, $M = (\begin{psmallmatrix} a & b \\ c & d \end{psmallmatrix}, \sqrt{c \tau + d})$ is any element with bottom row $c,d$. The most important case is the \textbf{Poincar\'e series of exponential type}: for $\beta \in A$ and $m \in \mathbb{Z} - Q(\beta)$, we take the seed function $\phi(\tau) = q^m \frac{ \mathfrak{e}_{\beta} + (-1)^k \mathfrak{e}_{-\beta}}{2}$ and define $P_{k,m,\beta} = \mathbb{P}_k(\phi)$. These converge normally when $k \ge 5/2$ and are cusp forms when $m > 0$. Moreover they are a spanning set of $S_k(\rho^*)$ as $\beta$ runs through $A$ and as $m$ runs through $(\mathbb{Z} - Q(\beta))_{>0}$ because, with respect to the Petersson scalar product $$(f,g) = \int_{\mathrm{SL}_2(\mathbb{Z}) / \mathbb{H}} \langle f(\tau), g(\tau) \rangle y^{k-2} \, \mathrm{d}x \, \mathrm{d}y, \; \; f,g \in S_k(\rho^*),$$ these Poincar\'e series satisfy $$(f,P_{k,m,\beta}) = (4\pi m)^{1-k} \Gamma(k-1) c(m,\beta) \; \text{for all} \; f(\tau) = \sum_{\gamma \in A} \sum_{n \in \mathbb{Z} - Q(\gamma)} c(n,\gamma) q^n \mathfrak{e}_{\gamma} \in S_k(\rho^*).$$ (This is proved by the usual Rankin-Selberg unfolding argument.) In particular any cusp form orthogonal to all $P_{k,m,\beta}$ is identically zero.

\section{Antisymmetric Poincar\'e series with rational coefficients}

Fix a discriminant form $(A,Q)$ and a weight $k \ge 7/2$ which is antisymmetric, i.e. $k + \mathrm{sig}(A,Q)/2$ is odd. For an index $(m,\beta)$ with $\beta \in A$ and $m \in \mathbb{Z} - Q(\beta)$, let $P_{k,m,\beta}$ be the Poincar\'e series of exponential type of weight $k$ as in section 2.

\begin{lem} The series $R_{k,m,\beta} = \sum_{\lambda \in \mathbb{Z}} \lambda P_{k,\lambda^2 m, \lambda \beta} = 2 \cdot \sum_{\lambda = 1}^{\infty} \lambda P_{k,\lambda^2 m,\lambda \beta}$ converges in $S_k(\rho^*)$.
\end{lem}
\begin{proof} $S_k(\rho^*)$ is finite-dimensional so every reasonable notion of convergence (e.g. with respect to any norm) coincides with that of the weak topology with respect to the Petersson scalar product; in other words, it is enough to show that $$\sum_{\lambda = 1}^{\infty} \lambda (f,P_{k,\lambda^2 m, \lambda \beta}) = (4\pi m)^{1-k} \Gamma(1-k) \sum_{\lambda = 1}^{\infty} \frac{c(\lambda^2 m, \lambda \beta)}{\lambda^{2k-3}}$$ converges for all cusp forms $f(\tau) = \sum_{n,\gamma} c(n,\gamma) q^n \mathfrak{e}_{\gamma}$. For $k \ge 7/2$ this follows from known coefficient bounds for cusp forms. See also the analogous argument in remark 10 of \cite{W}.
\end{proof}

\begin{lem} The series $R_{k,m,\beta}$ span $S_k(\rho^*)$ as $\beta$ runs through $A$ and $m$ runs through positive elements of $\mathbb{Z} - Q(\beta)$.
\end{lem}
\begin{proof} Let $\mu$ be the M\"obius function. M\"obius inversion implies $$P_{k,m,\beta} = \frac{1}{2} \sum_{\lambda = 1}^{\infty} \lambda \mu(\lambda) R_{k,\lambda^2 m, \lambda \beta}$$ with convergence by the same argument as the previous lemma. Since $S_k(\rho^*)$ is finite-dimensional it follows from this that all Poincar\'e series $P_{k,m,\beta}$ lie in the span of $R_{k,m,\beta}$.
\end{proof}

The most practial way to compute $R_{k,m,\beta}$ is to view it as a development coefficient of a (vector-valued) Jacobi Eisenstein series of weight $k-1$. For now we assume $k \ge 9/2$. The cases $k \in \{5/2,3,7/2,4\}$ are treated in the next section. \\

A \textbf{Jacobi form} of weight $k$ and index $(m,\beta)$ is a holomorphic function of two variables $$\Phi : \mathbb{H} \times \mathbb{C} \longrightarrow \mathbb{C}$$ which satisfies the transformation laws $$\Phi\Big( \frac{a \tau + b}{c \tau + d}, \frac{z}{c \tau +d}  \Big) = (c \tau + d)^k \mathbf{e}\Big( \frac{mcz^2}{c \tau + d} \Big) \rho^*(M) \Phi(\tau,z) \; \text{for all} \; M = (\begin{psmallmatrix} a & b \\ c & d \end{psmallmatrix}, \sqrt{c \tau + d}) \in \tilde \Gamma$$ and $$\Phi(\tau,z+\lambda \tau + \mu) = \mathbf{e}(-\lambda \mu m) q^{-m \lambda^2} \zeta^{-2m\lambda} \sigma_{\beta}^*(\lambda, \mu) \Phi(\tau,z) \; \text{for all} \; \lambda, \mu \in \mathbb{Z},$$ together with a vanishing condition on Fourier coefficients, where $q = e^{2\pi i \tau}$ and $\zeta = e^{2\pi i z}$ and where $$\sigma_{\beta}^*(\lambda,\mu) \mathfrak{e}_{\gamma} = \mathbf{e}\Big( - \mu \langle \beta, \gamma \rangle + \lambda \mu Q(\beta) \Big) \mathfrak{e}_{\gamma - \lambda \beta}.$$ The maps $\sigma_{\beta}^*$ do not define a representation of $\mathbb{Z}^2$ but they can be interpreted as a sort of finite analogue of the Schr\"odinger representations of the integral Heisenberg group $\mathcal{H}$. In particular Jacobi forms $\Phi$ as in the definition above are automorphic under the Jacobi group $\mathcal{J} = \mathcal{H} \rtimes Mp_2(\mathbb{Z})$ with respect to the semidirect product representation $\rho_{\beta}^* = \sigma_{\beta}^* \rtimes \rho^*$. (This is explained in more detail in section 3 of \cite{W}.) Any Jacobi form of index $(m,\beta)$ has a Fourier expansion of the form $$\Phi(\tau,z) = \sum_{\gamma \in A} \sum_{n \in \mathbb{Z} - Q(\gamma)} \sum_{r \in \mathbb{Z} - \langle \gamma, \beta \rangle} c(n,r,\gamma) q^n \zeta^r \mathfrak{e}_{\gamma}, \; \; c(n,r,\gamma) \in \mathbb{C},$$ and the vanishing condition on Fourier coefficients is that $c(n,r,\gamma) = 0$ whenever $4mn - r^2 < 0$. \\

The \textbf{Jacobi Eisenstein series} $E_{k,m,\beta}$ of weight $k$ and index $(m,\beta)$ is obtained by averaging out the constant function $\mathfrak{e}_0$ to a Jacobi form of that weight and index: \[ E_{k,m,\beta}(\tau,z) = \frac{1}{2} \sum_{\substack{c,d \in \mathbb{Z} \\ \mathrm{gcd}(c,d) = 1}} (c \tau + d)^{-k} \sum_{\lambda \in \mathbb{Z}} \mathbf{e}\Big( \frac{m \lambda^2 (a \tau + b) + 2m \lambda z - cmz^2}{c \tau + d} \Big) \rho^*\left( \begin{psmallmatrix} a & b \\ c & d \end{psmallmatrix}, \sqrt{c \tau + d} \right)^{-1} \mathfrak{e}_{\lambda \beta}, \] where $\begin{psmallmatrix} a & b \\ c & d \end{psmallmatrix} \in \mathrm{SL}_2(\mathbb{Z})$ is any matrix with bottom row $(c,d)$. This converges and defines a Jacobi form when $k \ge 3$ (and it is zero unless $k$ is a symmetric weight for $\rho^*$). \\

\begin{lem} Suppose $k \ge 9/2$ is an antisymmetric weight. Then $$R_{k,m,\beta}(\tau) = \frac{1}{4\pi m i} \frac{\partial}{\partial z} \Big|_{z=0} E_{k-1,m,\beta}(\tau,z).$$
\end{lem}

\begin{proof} The triple series over $(c,d,\lambda)$ in the definition of $E_{k-1,m,\beta}$ converges normally (as remarked in \cite{W}) so we swap the order of summation and find \begin{align*} &\quad \frac{\partial}{\partial z} \Big|_{z=0} E_{k-1,m,\beta}(\tau,z) \\ &= \sum_{\lambda \in \mathbb{Z}} \sum_{c,d} (c \tau + d)^{1-k} \frac{\partial}{\partial z} \Big|_{z=0} \mathbf{e}\Big( \frac{m \lambda^2 (a \tau + b) + 2m \lambda z - cmz^2}{c \tau + d} \Big) \rho^*\left( \begin{psmallmatrix} a & b \\ c & d \end{psmallmatrix}, \sqrt{c \tau + d} \right)^{-1} \mathfrak{e}_{\lambda \beta} \\ &= 4\pi m i \sum_{\lambda \in \mathbb{Z}} \sum_{c,d} \lambda (c \tau + d)^{-k} \mathbf{e}\Big( m \lambda^2 \frac{a \tau + b}{c \tau + d} \Big) \rho^*\left( \begin{psmallmatrix} a & b \\ c & d \end{psmallmatrix}, \sqrt{c \tau + d} \right)^{-1} \mathfrak{e}_{\lambda \beta}  \\ &= 4\pi m i R_{k,m,\beta}. \qedhere \end{align*}
\end{proof}

In particular, if we write out the Fourier expansion of the Jacobi Eisenstein series as $$E_{k-1,m,\beta}(\tau,z) = \sum_{\gamma \in A} \sum_{n \in \mathbb{Z} - Q(\gamma)} \sum_{r \in \mathbb{Z} - \langle \gamma, \beta \rangle} c(n,r,\gamma) q^n \zeta^r \mathfrak{e}_{\gamma}$$ then we find the Fourier expansion \begin{align*} R_{k,m,\beta}(\tau) &= \frac{1}{2m} \sum_{\gamma \in A} \sum_{n \in \mathbb{Z} - Q(\gamma)} \sum_{r \in \mathbb{Z} - \langle \gamma, \beta \rangle} c(n,r,\gamma) q^n \mathfrak{e}_{\gamma} \frac{1}{2\pi i} \frac{\partial}{\partial z} \Big|_{z=0} \zeta^r \\ &= \sum_{\gamma \in A} \sum_{n \in \mathbb{Z} - Q(\gamma)} \Big( \frac{1}{2m} \sum_{\substack{r \in \mathbb{Z} - \langle \gamma, \beta \rangle \\ r^2 \le 4mn}} r c(n,r,\gamma) \Big) q^n \mathfrak{e}_{\gamma}. \end{align*}

To complete the proofs of Theorems 1 and 2 of the introduction (in weights $k \ge 9/2$), we use the theta decomposition (Theorem 5.1 of \cite{EZ}) to relate the coefficients of the Jacobi Eisenstein series to the usual (modular) Eisenstein series. The proof of Eichler and Zagier \cite{EZ} does not immediately apply to Jacobi forms of non-integral index but a minor extension (\cite{W3}; see also the errata) is sufficient. Following \cite{W3}, if $\Lambda$ is an even lattice which realizes the discriminant form $(A,Q)$ and $\Lambda_{m,\beta}$ is defined as in the statement of Theorem 1.2 above, then there is an isomorphism $$\Theta : M_{k-3/2}(\rho^*_{\Lambda_{m,\beta}}) \stackrel{\sim}{\longrightarrow} J_{k-1,m,\beta}(\rho^*_{\Lambda})$$ (of modules over the graded ring of classical modular forms of level one) which identifies a vector-valued modular form $$F(\tau) = \sum_{\gamma \in \Lambda_{m,\beta}'/\Lambda_{m,\beta}} \sum_{n \in \mathbb{Z} - Q_{m,\beta}(\gamma)} c(n,\gamma) q^n \mathfrak{e}_{\gamma} \in M_{k-3/2}(\rho^*_{\Lambda_{m,\beta}})$$ with the Jacobi form $$\Phi(\tau,z) = \sum_{\substack{\gamma \in \Lambda'/\Lambda \\ n \in \mathbb{Z} - Q(\gamma) \\ r \in \mathbb{Z} - \langle \gamma, \beta \rangle}} c \Big( n - \frac{r^2}{4m}, (\gamma - \frac{r}{2m} \beta, \frac{r}{2m} \beta) \Big) q^n \zeta^r \mathfrak{e}_{\gamma}.$$ Moreover, this correspondence identifies vector-valued Eisenstein series with Jacobi Eisenstein series (\cite{W3}, section 5). This immediately yields Theorem 1.2; and the rationality result (Theorem 1.1) follows from rationality of the Fourier coefficients of vector-valued Eisenstein series (which follows from the formula \cite{BK}).

\section{Low weights}

In weights $5/2 \le k \le 4$ it is still possible to define a cusp form $R_{k,m,\beta}$ through the identity $$(f,R_{k,m,\beta}) = (4\pi m)^{1-k} \Gamma(1-k) \sum_{\lambda = 1}^{\infty} \frac{c(\lambda^2 m, \lambda \beta)}{\lambda^{2k-3}}, \; \text{for all cusp forms} \; f = \sum_{m,\beta} c(m,\beta) q^m \mathfrak{e}_{\beta} \in S_k(\rho^*),$$ where the $L$-value $\sum_{\lambda = 1}^{\infty} \frac{c(\lambda^2 m, \lambda \beta)}{\lambda^{2k-3}}$ is defined by analytic continuation if necessary. We will compute these series using the Jacobi Eisenstein series \begin{align*} &\quad E_{k-1,m,\beta}^*(\tau,z;s) \\ &= \frac{1}{2} \sum_{\substack{c,d \in \mathbb{Z} \\ \mathrm{gcd}(c,d) = 1}} \sum_{\lambda \in \mathbb{Z}} \frac{y^s}{(c \tau + d)^{k+s-1} (c \overline{\tau} + d)^s} \mathbf{e}\Big( \frac{m \lambda^2 (a \tau + b) + 2m \lambda z - cmz^2}{c \tau + d} \Big) \rho^*\left( \begin{psmallmatrix} a & b \\ c & d \end{psmallmatrix}, \sqrt{c \tau + d} \right)^{-1} \mathfrak{e}_{\lambda \beta},\end{align*} which is well-defined for $\mathrm{Re}[s]$ sufficiently large and has an analytic continuation to $s=0$. (If $k > 1$ then the analytic continuation is given by continuing each Fourier coefficient to $s=0$ separately.) The zero-value $E_{k-1,m,\beta}^*(\tau,z,0)$ is not necessarily holomorphic, but one can obtain a well-defined series $R_{k,m,\beta}$ by replacing $\frac{1}{4\pi m i} \frac{\partial}{\partial z} \Big|_{z=0} E_{k-1,m,\beta}(\tau,z)$ by its holomorphic projection, i.e. its orthogonal projection with respect to the Petersson scalar product to the space of holomorphic cusp forms. The identity $(f,R_{k,m,\beta}) = (4\pi m)^{1-k} \Gamma(1-k) \sum_{\lambda = 1}^{\infty} \frac{c(\lambda^2 m, \lambda \beta)}{\lambda^{2k-3}}$ is clear. \\

We should point out that for $k \in \{7/2,4\}$ this agrees with the definition (Lemma 3.1) by essentially the same proof as that of Lemma 3.4: since one can swap the order of summation and differentiate termwise freely for large enough $\mathrm{Re}[s]$, and use the uniqueness of analytic continuation to $s=0$. \\

It will be convenient to denote by $\tilde R_{k,m,\beta}$ the $q$-series obtained in Theorem 1.2 when one takes as $c_{m,\beta}(n,\gamma)$ the result of Bruinier and Kuss's formula \cite{BK} for the Eisenstein series naively evaluated in low weights (where it generally does not define a modular form). To be explicit we can use the formula of Proposition 4.3 of \cite{BK} in the following form. Suppose $(A,Q)$ is represented by an $\ell$-dimensional even lattice $\Lambda$. Then $E_k(\tau) = \sum_{n,\gamma} c(n,\gamma) q^n \mathfrak{e}_{\gamma}$ with $$c(n,\gamma) = \frac{(-2\pi i)^k n^{k-1} \mathbf{e}(-\mathrm{sig}(A,Q)/8)}{\Gamma(k) \sqrt{|A|} } \tilde L(n,\gamma,k+\ell/2)$$ where $\tilde L(n,\gamma,s)$ is the $L$-series $$\tilde L(n,\gamma,s) = \zeta(s - \ell)^{-1} \sum_{c=1}^{\infty} \mathbf{N}_{n,\gamma}(c) c^{1-s}, \; \text{where} \; \mathbf{N}_{n,\gamma}(c) = \#\{x \in \Lambda/c\Lambda: \; Q(x-\gamma) + n \equiv 0 \; \text{mod} \; c\}.$$  (Unlike \cite{BK} we have normalized such that $c(0,0) = 1$.) These Eisenstein series are also discussed in \cite{W4} in more detail. \\

\begin{prop} Suppose $k = 4$. Then $R_{4,m,\beta} = \tilde R_{4,m,\beta}$; i.e. Theorem 1.1 and Theorem 1.2 hold with no modifications.
\end{prop}
This result is not very surprising and the proof is essentially the same as the propositions below, so we omit the details. Ultimiately the reason why no modification is necessary is the fact that all Eisenstein series in weight $k-3/2 = 5/2$ (including the ones appearing in Theorem 1.2) are holomorphic modular forms.

\begin{prop} Suppose $k = 7/2$. Then $$R_{7/2,m,\beta} = \tilde R_{7/2,m,\beta} - \frac{1}{3\pi m} \frac{\partial^2}{\partial \tau \partial z} \Big|_{z=0} \vartheta(\tau,z),$$ where $\vartheta(\tau,z)$ is the unique weight $1/2$ Jacobi form for which $E_{5/2,m,\beta}^*(\tau,z,0) + y^{-1} \vartheta(\tau,z)$ is holomorphic.
\end{prop}
A formula for the Fourier coefficients of $\vartheta$ is given in section 7 of \cite{W}: $$\vartheta(\tau,z) = \sum_{\substack{n,r,\gamma \\ r^2 = 4mn}} a(n,r,\gamma) q^n \zeta^r \mathfrak{e}_{\gamma},$$ where $$a(n,r,\gamma) = (-1)^{\frac{5 + \mathrm{sig}(A,Q)}{4}} \cdot \frac{\pi}{2 \sqrt{2m|A|}} \cdot \mathrm{Res}\Big( \tilde L(n,\gamma,s+5/2+\ell/2); s = 0 \Big).$$ From \cite{W} it follows that $\tilde L(n,\gamma,s+(5/2+\ell/2))$ equals $\frac{\zeta(s+1)}{\zeta(s+2)}$ up to finitely many Euler factors which are rational at $s=0$. In particular the coefficients of $\vartheta$ lie in $\frac{1}{\pi} \cdot \mathbb{Q}$. (This should be compared with the modular correction of the classical Eisenstein series of weight two: $E_2^*(\tau) = E_2(\tau) - (3/\pi)y$.) Therefore all Fourier coefficients of $R_{7/2,m,\beta}$ are rationals.

\begin{proof} If we write out the Fourier expansions $$\vartheta(\tau,z) = \sum_{\gamma,n,r} a(n,r,\gamma) q^n \zeta^r \mathfrak{e}_{\gamma}, \; \; E_{5/2,m,\beta}^*(\tau,z,0) = \sum_{\gamma,n,r} \Big( c(n,r,\gamma) + y^{-1} a(n,r,\gamma) \Big) q^n \zeta^r \mathfrak{e}_{\gamma}$$ then the Fourier expansion of $R_{7/2,m,\beta}(\tau) = \sum_{n,\gamma} b(n,\gamma) q^n \mathfrak{e}_{\gamma}$ may be found using the Rankin-Selberg method: \begin{align*} b(n,\gamma) &= (4\pi n)^{5/2} \Gamma(5/2)^{-1} (R_{7/2,m,\beta},P_{7/2,n,\gamma}) \\ &= (4\pi n)^{5/2} \Gamma(5/2)^{-1} \Big( \frac{1}{2\pi i} \frac{\partial}{\partial z} \Big|_{z=0} E_{5/2,m,\beta}^*(\tau,z,0), P_{7/2,n,\gamma} \Big) \\ &= \frac{64 \pi^2 n^{5/2}}{3m} \sum_{r \in \mathbb{Z} - \langle \gamma, \beta \rangle} \int_0^{\infty} r \cdot \Big( (c(n,r,\gamma) + y^{-1} a(n,r,\gamma) \Big) e^{-4\pi n y} y^{3/2} \, \mathrm{d}y \\ &= \sum_{r \in \mathbb{Z} - \langle \gamma, \beta \rangle} r \cdot \Big( \frac{1}{2m} c(n,r,\gamma) + \frac{4\pi n}{3m} a(n,r,\gamma) \Big), \end{align*} whereas $\tilde R_{7/2,m,\beta}(\tau,z) = \frac{1}{2m} \sum_{\gamma,n} \sum_{r \in \mathbb{Z} - \langle \gamma, \beta \rangle} r c(n,r,\gamma) q^n \mathfrak{e}_{\gamma}.$
\end{proof}

\begin{prop} Suppose $k = 3$. Then all Fourier coefficients of $R_{3,m,\beta}$ are rational, and the collection of $(R_{3,m,\beta})_{m,\beta}$ as $(m,\beta)$ runs through valid indices spans $S_3(\rho^*)$.
\end{prop}
\begin{proof} By \cite{W2} one has the decomposition $$E_{2,m,\beta}^*(\tau,z,0) = E_{2,m,\beta}(\tau,z) + y^{-1/2} \sum_{\gamma \in A} \sum_{\substack{n \in \mathbb{Z} - Q(\gamma) \\ n > 0}} \sum_{\substack{r \in \mathbb{Z} - \langle \gamma, \beta \rangle \\ r^2 \ge 4mn}} a(n,r,\gamma) \beta\Big( \frac{\pi y (r^2 - 4mn)}{m} \Big) q^n \zeta^r \mathfrak{e}_{\gamma},$$ where $E_{2,m,\beta}(\tau,z)$ is holomorphic with rational Fourier coefficients, $\beta(t) = \frac{1}{16\pi} \int_1^{\infty} u^{-3/2} e^{-tu} \, \mathrm{d}u$ is an incomplete Gamma function, and where $a(n,r,\gamma)$ are coefficients lying in $\frac{1}{\sqrt{m|A|}} \cdot \mathbb{Q}$ which are zero unless $\sqrt{|A|(r^2 - 4mn)} \in \mathbb{Q}$ (and, when nonzero, can be computed effectively). As before, we apply the Rankin-Selberg method to obtain the Fourier expansion $R_{3,m,\beta}(\tau) = \sum_{\gamma,n} b(n,\gamma) q^n \mathfrak{e}_{\gamma}$ with \begin{align*} b(n,\gamma) &= \frac{(4\pi n)^2}{\Gamma(2)} (R_{3,m,\beta},P_{3,n,\gamma}) \\ &= 16 \pi^2 n^2 \Big( \frac{1}{2\pi i} \frac{\partial}{\partial z} \Big|_{z=0} E_{2,m,\beta}^*(\tau,z,0), P_{3,n,\gamma} \Big) \\ &= 16\pi^2 n^2 \sum_{r \in \mathbb{Z} - \langle \gamma, \beta \rangle} \int_0^{\infty} r \cdot \Big( c(n,r,\gamma) + y^{-1/2} a(n,r,\gamma) \beta(\pi m^{-1} y (r^2 - 4mn)) \Big) e^{-4\pi n y} y \, \mathrm{d}y \\ &= \frac{1}{2m} \sum_r r c(n,r,\gamma) + (4\pi  n)^2 \sum_r r a(n,r,\gamma) \int_0^{\infty} e^{-4\pi n y} \beta(\pi y (r^2 / m - 4n)) y^{1/2} \, \mathrm{d}y. \end{align*}

Here $\frac{1}{2m} \sum_r r c(n,r,\gamma)$ is the coefficient of $\tilde R_{3,m,\beta}$. The integral above can be easily computed by reversing the order of integration:
\begin{align*} \int_0^{\infty} e^{-4\pi n y} \beta(\pi y (r^2 / m - 4n)) y^{1/2} \, \mathrm{d}y &= \frac{1}{16\pi} \int_0^{\infty} \int_1^{\infty} u^{-3/2} y^{1/2} e^{4\pi n y (u-1) - \pi r^2 y u / m}\, \mathrm{d}u \, \mathrm{d}y \\ &= \frac{1}{16\pi} \int_1^{\infty} u^{-3/2} [(r^2 / m - 4n)u + 4n]^{-3/2} \, \mathrm{d}u \\ &= \frac{1}{16\pi |r|} (|r| - \sqrt{r^2 - 4mn})^2, \end{align*} so altogether the coefficient of $q^n \mathfrak{e}_{\gamma}$ in the correction $R_{3,m,\beta} - \tilde R_{3,m,\beta}$ is $$\frac{1}{32 m^{3/2}} \sum_{r \in \mathbb{Z} - \langle \gamma, \beta \rangle} \mathrm{sgn}(r) a(n,r,\gamma) (|r| - \sqrt{r^2 - 4mn})^2.$$

The sum over $r$ above is generally an infinite series but it can be computed in exact form using the a similar argument to section 7 of \cite{W2}. If the discriminant $|A|$ is square, then the sum is actually finite (as there are only finitely many $r \in \mathbb{Z} - \langle \gamma, \beta \rangle$ for which $r^2 - 4mn$ is a perfect square) and the series can be summed directly. Moreover, each term $\frac{a(n,r,\gamma)}{m^{3/2}}$ and $(|r| - \sqrt{r^2 - 4mn})^2$ is rational, so the correction is a finite sum of rational numbers and therefore rational. \\

Otherwise, let $d_{\beta}$ and $d_{\gamma}$ be the denominators of $\beta$ and $\gamma$ (that is, the smalllest positive integers such that $d_{\beta} \cdot \beta$ and $d_{\gamma} \cdot \gamma$ are zero in $A$), and let $K = \mathbb{Q}(\sqrt{|A|})$ with ring of integers $\mathcal{O}_K$. The main point of section 7 of \cite{W2} is that there are finitely many algebraic integers $\mu_i$, $i = 1,...,N$ (``congruent fundamental solutions" to a Pell-type equation) and finitely many units $\varepsilon_i \in \mathcal{O}_K^{\times}$ such that, as $r$ runs through the numbers $\mathbb{Z} \pm \langle \gamma, \beta \rangle$ for which $(r^2 - 4mn)|A|$ is a rational square, $d_{\beta} d_{\gamma} (|r| - \sqrt{r^2 - 4mn})$ runs through $$\{\mu_i, \; \mu_i \varepsilon_i^{-n}, \; \mu_i' \varepsilon_i^{-n}: \; i = 1,...,N, \; n \in \mathbb{N}\}$$ (where $\mu_i'$ denotes the conjugate of $\mu_i$ in $K$) exactly once, with two exceptions: that $r^2 - 4mn = 0$ has a solution with $r \in \mathbb{Z} \pm \langle \gamma, \beta \rangle$ (in which case $d_{\beta}d_{\gamma} (|r| - \sqrt{r^2 - 4mn})$ takes the value $d_{\beta} d_{\gamma} |r|$ twice) or that $\mu_i' \in \mu_i \cdot \mathcal{O}_K$ (in which case each element in the multiset $\{\mu_i, \mu_i \varepsilon_i^{-n}, \mu_i' \varepsilon_i^{-n}\}$ appears twice and the result needs to be divided by two). Moreover, the modified coefficient $$a_i = a(n,r,\gamma) \times \begin{cases} 1 : & r^2 \ne 4mn; \\ 2: & r^2 = 4mn; \end{cases}$$ depends only on the index $i$ of $d_{\beta}d_{\gamma} (|r| - \sqrt{r^2 - 4mn})$ as an element of $\{\mu_i, \mu_i \varepsilon_i^{-n}, \, \mu_i' \varepsilon_i^{-n}\}$, and the sign $\mathrm{sgn}(r)$ equals $(-1)^n$ where $-n$ is the exponent of $\varepsilon_i$. \\

With that in mind, one can compute the correction term, using the antisymmetry of $R_{3,m,\beta}$ and $\tilde R_{3,m,\beta}$ under $\mathfrak{e}_{\gamma} \mapsto \mathfrak{e}_{-\gamma}$: \begin{align*} &\quad \frac{1}{32m^{3/2}}\sum_{r \in \mathbb{Z} - \langle \gamma, \beta \rangle} \mathrm{sgn}(r) a(n,r,\gamma) \Big( |r| - \sqrt{r^2 - 4mn} \Big)^2 \\ &= \frac{1}{64m^{3/2}} \sum_{r \in \mathbb{Z} - \langle \pm \gamma, \beta \rangle} \mathrm{sgn}(r) a(n,r,\pm \gamma) \Big( |r| - \sqrt{r^2 - 4mn} \Big)^2 \\ &= \frac{1}{64m^{3/2}}\sum_{i=1}^N \frac{a_i}{d_{\beta} d_{\gamma}} \Big( \mu_i^2 + (\mu_i^2 + (\mu_i')^2) \sum_{n=1}^{\infty} (-\varepsilon_i^2)^{-n} \Big) \times \begin{cases} 1: & \mu_i' \not \in \mu_i \mathcal{O}_K; \\ 1/2: & \mu_i' \in \mu_i \mathcal{O}_K; \end{cases} \\ &= \frac{1}{64m^{3/2}}\sum_{i=1}^N \frac{a_i}{d_{\beta} d_{\gamma} N_{K/\mathbb{Q}}(1 + \varepsilon_i^2)} \Big( \mu_i^2 - (\mu_i')^2 + (\mu_i \varepsilon_i)^2 - (\mu_i' \varepsilon_i')^2 \Big) \times \begin{cases} 1: & \mu_i' \not \in \mu_i \mathcal{O}_K; \\ 1/2: & \mu_i' \in \mu_i \mathcal{O}_K. \end{cases} \end{align*} This is a rational number because each $a_i$ lies in $\sqrt{m|A|} \cdot \mathbb{Q}$ and because each $\mu_i^2 - (\mu_i')^2$ and $(\mu_i \varepsilon_i)^2 - (\mu_i' \varepsilon_i')^2$ lies in $\sqrt{|A|} \cdot \mathbb{Q}$. \\

To see that $(R_{3,m,\beta})_{(m,\beta)}$ span $S_3(\rho^*)$ one can argue similarly to section 3. It is helpful to deform the Poincar\'e series to $$P_{k,m,\beta}^*(\tau,s) = \sum_{M \in \tilde \Gamma_{\infty} \backslash \tilde \Gamma} (y^s q^m \mathfrak{e}_{\beta}) \Big|_{k,\rho^*} M$$ (which is not the usual real-analytic deformation, as $y^s q^m$ is not harmonic!) and to use M\"obius inversion to see that $$P_{3,m,\beta}^*(\tau,s) - \frac{1}{2} \sum_{\lambda = 1}^{\infty} \lambda \mu(\lambda) \cdot \frac{1}{4\pi m i} \frac{\partial}{\partial z} \Big|_{z=0} E^*_{2,m,\beta}(\tau,z,s)$$ is well-defined and orthogonal to all cusp forms in $S_3(\rho^*)$ for every $\mathrm{Re}[s] > 0$. (For convergence in the range $\mathrm{Re}[s] > 0$ we need the full strength of the Deligne bound: the coefficients $c(n,\gamma)$ of a cusp form of weight three have growth bounded by $O(n^{1+\varepsilon})$ for every $\varepsilon > 0$.) Taking the limit as $s \rightarrow 0$ shows that $$P_{3,m,\beta} = \frac{1}{2} \sum_{\lambda=1}^{\infty} \lambda \mu(\lambda) R_{3,m,\beta} \in \overline{\mathrm{Span}(R_{3,m,\beta})_{m,\beta}} = \mathrm{Span}(R_{3,m,\beta})_{m,\beta}$$ so the series $R_{3,m,\beta}$ span all Poincar\'e series and therefore all of $S_3(\rho^*)$.
\end{proof}

\begin{ex} Suppose $(A,Q)$ is the cyclic quadratic module $A = \frac{1}{N}\mathbb{Z}/\mathbb{Z}$, $Q(x) = -Nx^2 + \mathbb{Z}$ with $N \equiv 1 \, (4)$. Then $\mathrm{sig}(A,Q) \equiv 0$ mod $8$ so $k=3$ is an antisymmetric weight. It was pointed out in \cite{W5} that the Jacobi Eisenstein series of weight two and the smallest index $(m,\beta) = (1/N,1/N)$ has the Fourier expansion

\begin{align*} E_{2,1/N,1/N}^*(\tau,z) &= \sum_{\gamma \in A} \sum_{n \in \mathbb{Z}-Q(\gamma)} \sum_{r \in \mathbb{Z} - \langle \gamma, 1/N \rangle} 12H(4n - Nr^2) q^n \zeta^r \mathfrak{e}_{\gamma} \\ &+ y^{-1/2} \sum_{\gamma \in A} \sum_{n \in \mathbb{Z} - Q(\gamma)} \sum_{r \in \mathbb{Z} - \langle \gamma, 1/N \rangle} A(n,r,\gamma) \beta(\pi y (Nr^2 - 4n)) q^n \zeta^r \mathfrak{e}_{\gamma}, \end{align*} where $H(d)$ is the Hurwitz class number (the number of $\mathrm{SL}_2(\mathbb{Z})$-equivalence classes of binary quadratic forms of discriminant $d$, each form weighted by $2/w$ where $w$ is the size of its automorphism group), and where $$A(n,r,\gamma) = \begin{cases} -24: & Nr^2 - 4n = 0; \\ -48: & Nr^2 - 4n \; \text{is a nonzero square}; \\ 0: & \text{otherwise}. \end{cases}$$ (Bear in mind that $4n - Nr^2$ in the above sum is always integral even though $n$ and $r$ are not.) In particular, the Fourier coefficients of $R_{3,1/N,1/N}$ are sums over Hurwitz class numbers. When $N \in \{5,9\}$ one can compute $S_3(\rho^*_{(A,Q)}) = 0$ and the identity $R_{3,1/N,1/N} = 0$ yields identities for class number sums which are more difficult to prove by other arguments. \\

We will give the computation for $N = 5$. The coefficient of $q^n$ in the naive expression $\tilde R_{3,1/N,1/N}$ is $$-30 \sum_{r \in \mathbb{Z}} (r + 4/5) H(4n - 5r^2 - 8r - 4) \; \text{or} \; -30 \sum_{r \in \mathbb{Z}} (r- 2/5) H(4n - 5r^2 + 4r - 4)$$ when $\gamma$ equals $1/5$ or $2/5$ mod $\mathbb{Z}$, respectively (and their negatives when $\gamma$ equals $4/5$ or $3/5$). Here we set $H(d) = 0$ for $d < 0$. Therefore we need to compute the correction term $\frac{1}{32m^{3/2}} \sum_r \mathrm{sgn}(r) a(n,r,\gamma) (|r| - \sqrt{r^2 - 4mn})^2$. Let $K = \mathbb{Q}(\sqrt{5})$. The ``congruent fundamental units" $\varepsilon_i$ above are all $(\frac{1 + \sqrt{5}}{2})^4 = \frac{7 + 3 \sqrt{5}}{2}$ and for $\mu = a+b \sqrt{5} \in \mathcal{O}_K$, we can compute $$\mu^2 - (\mu')^2 + (\mu \varepsilon_i)^2 - (\mu' \varepsilon_i')^2 = -3 \sqrt{5} (7a^2 - 30ab + 35b^2).$$ The series $r \in \mathbb{Z} - \langle \pm \gamma, \beta \rangle$ is easier to compute when rephrased as a sum over ideals in $\mathcal{O}_K$: after some algebra we find \begin{align*} &\quad \frac{1}{32 (1/5)^{3/2}} \sum_{r \in \mathbb{Z} - \langle \gamma, 1/5 \rangle} \mathrm{sgn}(r) a(n,r,\gamma) (|r| - \sqrt{r^2- 4n/5})^2 \\ &= \frac{1}{5} \sum_{N(\mathfrak{a}) = 5n} \Big( 7 (c^2 - a^2) - 30(|cd| - |ab|) + 35(d^2 -b^2) \Big), \end{align*} where $\mathfrak{a}$ runs through ideals of $\mathcal{O}_K = \mathbb{Z}[\frac{1 + \sqrt{5}}{2}]$ and where $a + b \sqrt{5}, c + d \sqrt{5} \in \mathfrak{a}$ are generators with minimal positive trace $2a, 2c > 0$ satisfying the following congruences: \\ (1) if $n \in 4/5 + \mathbb{Z}$ then $a \equiv 3 \, (5)$ and $c \equiv 2 \, (5)$; \\ (2) if $n \in 1/5 + \mathbb{Z}$ then $a \equiv 1 \, (5)$ and $c \equiv 4 \, (5)$. \\ Altogether, by comparing coefficients in $R_{3,1/5,1/5} = 0$ we obtain the identities

$$\sum_{r \in \mathbb{Z}} (r + 4/5) H(4n - 5r^2 - 8r) = -\frac{1}{150} \sum_{N(\mathfrak{a}) = 5n+4} \Big( 7 (c^2 - a^2) - 30(|cd| - |ab|) + 35(d^2 - b^2) \Big),$$ $$ \sum_{r \in \mathbb{Z}} (r - 2/5) H(4n - 5r^2 + 4r) = -\frac{1}{150} \sum_{N(\mathfrak{a}) = 5n+1} \Big( 7 (c^2 - a^2) - 30 (|cd| - |ab|) + 35(d^2 - b^2) \Big),$$ where $n \in \mathbb{N}_0$, and where $a,b,c,d$ are defined as above for any ideal $\mathfrak{a}$ of $\mathbb{Z}[\frac{1+\sqrt{5}}{2}]$. For example, when $n = 3$, the left-hand side of the equations above are $$\sum_{r \in \mathbb{Z}} (r + 4/5) H(4n - 5r^2 - 8r) = \frac{-6H(8) - H(15) + 4H(12)}{5} = -\frac{8}{15},$$ $$\sum_{r \in \mathbb{Z}} (r - 2/5) H(4n - 5r^2 + 4r) = \frac{-7H(3) - 2H(12) + 3H(11) + 8H(0)}{5} = -\frac{8}{15}.$$ The ideals of norm $5n+4 = 19$ are $(2 \sqrt{5} - 1)$ and $(2 \sqrt{5} + 1)$, and they have minimal-trace generators $$a + b \sqrt{5} = 8 + 3 \sqrt{5}, \; c + d \sqrt{5} = \frac{9}{2} + \frac{1}{2} \sqrt{5} \in (2 \sqrt{5} - 1),$$ $$a + b \sqrt{5} = 8 - 3 \sqrt{5}, \; c + d \sqrt{5} = \frac{9}{2} - \frac{1}{2} \sqrt{5} \in (2 \sqrt{5} + 1)$$ satisfying the congruence conditions. In particular, we find $$-\frac{1}{150} \sum_{N(\mathfrak{a}) = 19}  \Big( 7 (c^2 - a^2) - 30(|cd| - |ab|) + 35(d^2 - b^2) \Big) = -\frac{2}{150} \Big( 7 \cdot (81/4 - 64) - 30 \cdot (9/4 - 24) + 35 \cdot (1/4 - 9) \Big) = -\frac{8}{15}.$$ The only ideal of norm $16$ in $\mathcal{O}_K$ is $\mathfrak{a} = (4)$ with minimal-trace generators $$a + b \sqrt{5} = 6 + 2 \sqrt{5}, \; \; c + d \sqrt{5} = 4 \in \mathfrak{a},$$ so we find $$-\frac{1}{150} \sum_{N(\mathfrak{a}) = 4} \Big( 7 (c^2 - a^2) - 30(|cd| - |ab|) + 35(d^2 - b^2) \Big) = -\frac{1}{150} \Big( 7 \cdot (4^2 - 6^2) - 30 \cdot (0 - 12) + 35 \cdot (0 - 2^2) \Big) = -\frac{8}{15}.$$

By treating $N = 9$ similarly and focusing on the components $\gamma = 1/3,2/3$ of norm zero one obtains the identity $$\sum_{r \equiv a \, (3)} r H(4n - r^2) = \chi_3(a) \varepsilon(n) \sum_{d | n} \chi_3(d) \min(d, n/d)^2, \; n \in \mathbb{N}_0, \; a \in \mathbb{Z}/3\mathbb{Z},$$ where $\chi_3(n) = \left( \frac{n}{3}\right)$ is the Kronecker symbol and where $$\varepsilon(n) = \begin{cases} -1: & n \equiv 0 \, (3); \\ 1/2: & n \not \equiv 0 \, (3). \end{cases}$$ Here the correction terms are finite sums of expressions involving $d = (3/2) (|r| - \sqrt{r^2 - 4n/9})$, which is always an integral divisor of $n$ that is less than or equal to $n/d = (3/2) (|r| + \sqrt{r^2 -4n/9})$.
\end{ex}

Finally we consider the construction in weight $k=5/2$.

\begin{prop} Suppose $k = 5/2$. Then $$R_{5/2,m,\beta}(\tau) = \frac{1}{4\pi m i} \frac{\partial}{\partial z} \Big|_{z=0} E_{3/2,m,\beta}^*(\tau,z,0)$$ is a holomorphic cusp form with rational Fourier coefficients and $R_{5/2,m,\beta} - \tilde R_{5/2,m,\beta}$ is a weight $3/2$ theta series.
\end{prop}
We should emphasize that $R_{5/2,m,\beta} - \tilde R_{5/2,m,\beta}$ is a cusp form but its multiplier system is not related to $\rho^*$. (Its components are essentially of the form $\sum_{n=1}^{\infty} nq^{n^2}$, up to rescaling.) Therefore $\tilde R_{5/2,m,\beta}$ is not generally a modular form at all.
\begin{proof} Section 5 of \cite{W2} points out that the zero-value $E_{3/2,m,\beta}^*(\tau,z,0)$ is always a holomorphic Jacobi form (which may be identically zero) that differs from the naive result $E_{3/2,m,\beta}(\tau,z)$ of the coefficient formula of \cite{W} by a weight 1/2 theta series $\vartheta(\tau,z) = \sum_{\substack{n,r,\gamma \\ r^2 = 4mn}} a(n,r,\gamma) q^n \zeta^r \mathfrak{e}_{\gamma}$. Therefore \begin{align*} R_{5/2,m,\beta}(\tau) &= \frac{1}{4\pi m i} \frac{\partial}{\partial z} \Big|_{z=0} E^*_{k-1,m,\beta}(\tau,z,0) \\ &= \frac{1}{4\pi m i} \frac{\partial}{\partial z} \Big|_{z=0} E_{k-1,m,\beta}(\tau,z,0) + \frac{1}{2m} \sum_{\substack{n,r,\gamma \\ r^2 = 4mn}} r a(n,r,\gamma) q^n \mathfrak{e}_{\gamma} \\ &= \tilde R_{5/2,m,\beta}(\tau) + \frac{1}{2m} \sum_r a(r^2/4m,r,\gamma) r q^{r^2 / 4m} \mathfrak{e}_{\gamma}, \end{align*} and each component of the correction $R_{5/2,m,\beta} - \tilde R_{5/2,m,\beta}$ is a weight $3/2$ theta series of the form $\sum_r r q^{r^2 / 4m}$.
\end{proof}

\begin{rem} The proof that the family $(R_{k,m,\beta})_{m,\beta}$ spans $S_k(\rho^*)$ as $(m,\beta)$ runs through all valid indices does \emph{not} carry over to $k=5/2$, even if one assumes the strongest possible (unproven) bounds on the coefficients of cusp forms in this weight, since $s=0$ no longer lies in the closure of the half-plane of absolute convergence. In fact the spanning claim is almost certainly false. A specific discriminant form which seems to be a counterexample is $A = \mathbb{Z}/26\mathbb{Z}$ with quadratic form $Q(x) = \frac{1}{52}x^2 + \mathbb{Z}$. Using the theta decomposition (\cite{EZ}, Theorem 5.1), we can construct a weight 5/2 cusp form for $\rho^*_{(A,Q)}$ from the unique (up to multiples) weight 3 Jacobi form of index 13: $$\phi_3(\tau,z) = \Big( 9 \sin(2\pi z) + 7 \sin(4\pi z) - 17 \sin(6\pi z) + 4 \sin(8\pi z) + 7 \sin(10 \pi z) - 5 \sin(12 \pi z) + \sin(14\pi z) \Big) q + O(q^2).$$ Some computations make it seem likely that all $R_{5/2,m,\beta}$ are identically zero. Nevertheless there are many other discriminant forms for which the $R_{5/2,m,\beta}$ construction is nontrivial and even yields spanning sets.
\end{rem}

\section{Theta lifts}

The most important application of antisymmetric vector-valued modular forms is as inputs into the additive theta lift. This is a useful way to construct orthogonal modular forms which generalizes a number of better-known lifts (including the Shimura, Doi-Naganuma and Saito-Kurokawa lifts). \\

We recall this in the following simplified situation, in which we consider only lattices which split a unimodular plane and consider Fourier expansions at the distinguished cusp corresponding to this splitting. Suppose $\mathbf{S}$ is a symmetric integral matrix with even diagonal and signature $(1,\ell-1)$ for some $\ell \in \mathbb{N}$, and let $\Lambda$ be the lattice $\mathbb{Z}^{\ell}$ with quadratic form $Q(x) = \frac{1}{2} x^T \mathbf{S} x$. Fix a cone $C \subseteq \mathbb{C}^{\ell}$ of vectors with positive norm under $Q$. Then the tube domain $$\mathbb{H}_{\Lambda} = \{z = x+iy \in \mathbb{C}^{\ell} : \; y \in C\}$$ is acted upon by the discriminant kernel $$\Gamma_{\Lambda} = \{M \in \mathrm{SO}^+(\Lambda \oplus II_{1,1}): \; M \; \text{acts trivially on} \; \Lambda'/\Lambda\}$$ by M\"obius transformations in a natural way that induces a cocycle $j(M;z)$. More explicitly one can realize $\Gamma_{\Lambda}$ as a subgroup of integral matrices which preserve the quadratic form $\begin{psmallmatrix} 0 & 0 & 1 \\ 0 & \mathbf{S} & 0 \\ 1 & 0 & 0 \end{psmallmatrix}$ and define $$M \cdot z = w \; \text{if and only if} \; M \begin{psmallmatrix} -Q(z) \\ z \\ 1 \end{psmallmatrix} = j(M;z) \begin{psmallmatrix} -Q(w) \\ w \\ 1 \end{psmallmatrix}.$$ A modular form is a holomorphic function $f : \mathbb{H}_{\mathbf{S}} \rightarrow \mathbb{C}$ satisfying the functional equations $f(M \cdot z) = j(M;z)^k f(z)$ and the usual growth condition at the cusps of $\Gamma_{\mathbf{S}} \backslash \mathbb{H}_{\mathbf{S}}$. \\

Every modular form $f(z)$ has a Fourier expansion which we write in the form $$f(z) = \sum_{\lambda \in \Lambda'} a(\lambda) \mathbf{q}^{\lambda}, \; \; \mathbf{q}^{\lambda} = e^{2\pi i \langle \lambda, z \rangle} = e^{2\pi i \lambda^T \mathbf{S} z},$$ where $a(\lambda) = 0$ unless $\langle \lambda, y \rangle \ge 0$ for all $y \in C$. \\

Let $k \ge 2$, $k \in \mathbb{N}$. To any cusp form $F(\tau) = \sum_{\gamma \in \Lambda'/\Lambda} \sum_{n \in \mathbb{Z} + Q(\gamma)} c(n,\gamma) q^n \mathfrak{e}_{\gamma}$ of weight $\kappa = k + 1 - \ell/2$ for the \emph{non-dual} Weil representation attached to $\Lambda$ (which is, with our definition, the Weil representation attached to $(\Lambda,-Q)$), the theta lift $$\Phi_F(z) = \sum_{\substack{\lambda \in \Lambda' \\ \langle \lambda, C \rangle > 0}} \sum_{n=1}^{\infty} c(Q(\lambda),\lambda) n^{k-1} \mathbf{q}^{n \lambda}$$ is a cusp form for $\Gamma_{\Lambda}$ of weight $k$. Moreover the map $F \mapsto \Theta_F$ is injective. In particular the forms $R_{k,m,\beta}$ may be lifted to orthogonal cusp forms of odd weight. \\

Following Oda \cite{O} (see also Borcherds \cite{Bn}), one can write $\Phi_F$ as the integral $$\Phi_F(z) = \frac{1}{2} (2iQ(v))^{-k} \int_{\mathrm{SL}_2(\mathbb{Z}) \backslash \mathbb{H}} \langle F(\tau), \Theta_k(\tau,z) \rangle y^{k-1} \, \mathrm{d}x \, \mathrm{d}y, \; \; z = u+iv \in \mathbb{H}_{\Lambda},$$ where $$\Theta_k(\tau,z) = \sum_{\substack{a,c \in \mathbb{Z} \\ b \in \Lambda'}} (aQ(z) + \langle b,z \rangle + c)^k e^{-\frac{\pi y}{Q(v)} |aQ(z) + \langle b,z \rangle + c|^2 + 2\pi i \overline{\tau} (Q(b) -ac)} \, \mathfrak{e}_b$$ is the theta kernel. When $F = P_{\kappa,m,\beta}$, $\kappa = k+1-\ell/2 \ge 2$ is the Poincar\'e series, the Rankin-Selberg unfolding method yields \begin{align*} \Phi_F(z) &= \frac{1}{2}(2iQ(v))^{-k} \sum_{\substack{a,c \in \mathbb{Z} \\ b \in \Lambda + \beta \\ Q(b) - ac = m}} (aQ(\overline{z}) + \langle b, \overline{z} \rangle + c)^k \int_0^{\infty} e^{-\pi y |aQ(z) + \langle b,z \rangle + c|^2 / Q(v)} y^{k-1} \, \mathrm{d}y \\ &= \frac{1}{2} (2iQ(v))^{-k} \sum_{a,b,c} \overline{(aQ(z) + \langle b,z \rangle + c)^k} \frac{(k-1)! Q(v)^k}{\pi^k |aQ(z) + \langle b,z \rangle + c|^{2k}} \\ &= \frac{(k-1)!}{2} (2\pi i)^{-k} \sum_{a,b,c} (aQ(z) + \langle b,z \rangle + c)^{-k}. \end{align*}

When $k$ is odd (and $\kappa$ is at least $7/2$, so $k \ge (\ell + 5)/2$) we can apply this to the series $F = R_{\kappa,m,\beta} = 2 \sum_{\lambda=1}^{\infty} P_{k,\lambda^2 m, \lambda \beta}$ by first lifting each Poincar\'e series and then summing over $\lambda$, i.e. \begin{align*} \Phi_F(z) &= (k-1)! (2\pi i)^{-k} \sum_{\substack{a,c \in \mathbb{Q} \\ b \in \Lambda \otimes \mathbb{Q} \\ Q(b) - ac = m}} \sum_{\substack{\lambda \in \mathbb{N}\\ \lambda (a,b,c) \in \mathbb{Z} \oplus (\Lambda + \beta) \oplus \mathbb{Z}}} \lambda^{1-k} (aQ(z) + \langle b,z \rangle + c)^{-k}. \end{align*} Finally, by replacing each tuple $(a,b,c,\lambda)$ in the expression above by $(a/\lambda,b/\lambda,c/\lambda,\lambda)$ we obtain the formula:

\begin{prop} The theta lift of $F = R_{\kappa,m,\beta}$ is $$\Phi_F(z) = (k-1)! (2\pi i)^{-k} \zeta(k-1) \sum_{\substack{a,c \in \mathbb{Q} \\ b \in \Lambda \otimes \mathbb{Q} \\ Q(b) - ac = m}} \mathrm{denom}(a,b-\beta,c)^{1-k} (aQ(z) + \langle b,z \rangle + c)^{-k},$$ where $d = \mathrm{denom}(a,b-\beta,c) \in \mathbb{N}$ is minimal such that $da,dc \in \mathbb{Z}$ and $d(b-\beta) \in \Lambda$, and where $\zeta(s)  = \sum_{n=1}^{\infty} n^{-s}$ is the Riemann zeta function.
\end{prop}

We will work out two examples in detail for lattices of signature $(2,1)$ and $(2,2)$ whose automorphic forms have classical interpretations as elliptic and Hilbert modular forms respectively.


\begin{ex} $\mathrm{PSL}_2(\mathbb{R})$ is identified with $\mathrm{SO}^+(2,1)$ through the adjoint action on its Lie algebra, and $\Gamma_0(N)$ preserves the lattice $\{\sqrt{N} \cdot \begin{psmallmatrix} x & y/N \\ z & -x \end{psmallmatrix}: \; x,y,z \in \mathbb{Z}\} \subseteq \mathfrak{sl}_2(\mathbb{R})$, which (with the Killing form) is isometric to $A_1(N) \oplus II_{1,1}$. Therefore, taking $\mathbf{S} = (2N)$, we obtain elliptic modular forms of level $\Gamma_0(N)$ as theta lifts. \\

When $N = 1$ there are no antisymmetric modular forms for $\rho^*_{\Lambda}$. For $N = 2$ one can identify the input vector-valued forms of antisymmetric weight $k \in 3/2 + 2\mathbb{Z}$ with scalar-valued modular forms of level $1$ and weight $k-3/2$ by the isomorphism $$M_{k-3/2} \stackrel{\sim}{\longrightarrow} S_k(\rho^*), \; \; f(\tau) \mapsto f(\tau) \eta(\tau)^3 (\mathfrak{e}_{1/4} - \mathfrak{e}_{3/4}),$$ where $\eta(\tau) = q^{1/24} \prod_{n=1}^{\infty} (1 - q^n)$ is the Dedekind eta function. For example, in weight $11/2$ our construction yields the input function $$R_{11/2,1/8,(3/4)}(\tau) = E_4(\tau) \eta(\tau)^3 (\mathfrak{e}_{1/4} - \mathfrak{e}_{3/4}) = (q^{1/8} + 237q^{9/8} + 1440q^{17/8} \pm ...) (\mathfrak{e}_{1/4} - \mathfrak{e}_{3/4}).$$ Writing this in the form $F(\tau) = \sum_{n=1}^{\infty} c(n) q^{n/8} (\mathfrak{e}_{1/4} - \mathfrak{e}_{3/4})$ we obtain the theta lift $$\Phi_F(z) = \sum_{n=1}^{\infty} \sum_{d|n} \Big( \frac{-1}{d} \Big) c(d^2) (n/d)^4 q^n = q + 16q^2 - 156q^3 + 256q^4 \pm ... = \eta(z)^8 \eta(2z)^8 (2 E_2(2z) - E_2(z))$$ where $E_2(z) = 1 - 24 \sum_{n=1}^{\infty} \frac{n q^n}{1 - q^n}$. By unfolding the lift as above we obtain the identity $$\sum_{\substack{a,b,c \in \mathbb{Q} \\ b^2 - 4ac = 1}} \mathrm{denom}\Big( \frac{a}{2}, \frac{b-3}{4}, c \Big)^{-4} (az^2 + bz + c)^{-5} = 120 \pi i \cdot \eta(z)^8 \eta(2z)^8 (2E_2(2z) - E_2(z)).$$

Of course one can use the forms $Q_{k,m,\beta}$ defined in \cite{W} to compute expressions of this type for even $k$. For example, when $N=1$, computing the preimage of the discriminant $\Delta$ under the Shimura lift yields the expression $$\Delta(\tau) = -\frac{691}{2^6 \cdot 3^5 \cdot 7^2} \sum_{\substack{a,b,c \in \mathbb{Q} \\ b^2 - 4ac = 1}} \mathrm{denom}\Big( a, \frac{b-1}{2}, c\Big)^{-6} (a\tau^2 + b \tau + c)^{-6}.$$
\end{ex}

\begin{ex} Suppose $K = \mathbb{Q}(\sqrt{D})$ is a real-quadratic number field of discriminant $D \equiv 1 \, (4)$. Let $\Lambda = \mathcal{O}_K$ be its lattice of integers with quadratic form given by the norm $N_{K/\mathbb{Q}}$. There is a well-known embedding of $\mathrm{PSL}_2(\mathcal{O}_K)$ into the orthogonal modular group $\Gamma_{\Lambda}$, under which translations $T_b = \begin{psmallmatrix} 1 & b \\ 0 & 1 \end{psmallmatrix}$ by $b \in \mathcal{O}_K$ correspond again to translations and under which $S = \begin{psmallmatrix} 0 & -1 \\ 1 & 0 \end{psmallmatrix}$ corresponds to the block matrix $\begin{psmallmatrix} 0 & 0 & -1 \\ 0 & J & 0 \\ -1 & 0 & 0 \end{psmallmatrix}$ where $J : \mathcal{O}_K \rightarrow \mathcal{O}_K$ is the orthogonal reflection along the vector $1$. (By a theorem of Vaserstein, $\mathrm{PSL}_2(\mathcal{O}_K)$ is generated by $S$ and $T_b$.) In this sense, orthogonal modular forms are the same as Hilbert modular forms which satisfy some extra symmetry conditions. \\

When $K$ has prime discriminant $p$ and the weight $k$ is even, Bruinier and Bundschuh \cite{BB} have given an equivalence between vector-valued modular forms (of symmetric weights) for the Weil representation attached to $(\Lambda,Q)$ and scalar-valued modular forms for $\Gamma_0(p)$ with the quadratic Nebentypus whose Fourier coefficients are supported on quadratic residues mod $p$. Through this interpretation the theta lift becomes the well-known Doi-Naganuma lift from $M_*(\Gamma_0(p))$ to $M_*(\mathrm{SL}_2(\mathcal{O}_K))$. This equivalence involves summing the components of the vector-valued modular forms and always yields zero in antisymmetric weights; in particular, one does not obtain Doi-Naganuma lifts in odd weight by lifting modular forms of level $\Gamma_0(p)$. (In view of \cite{SW} one can fix this by considering component sums which are twisted by odd Dirichlet characters, which yields a subspace of modular forms of level $\Gamma_0(p^2)$ which can be characterized through the Atkin-Lehner involution.) \\

In the simplest case we take the field $K = \mathbb{Q}(\sqrt{5})$. There is a unique (appropriately normalized) cusp form $s_5$ of weight $5$, which was constructed by Gundlach \cite{Gu} as the product of ten theta constants and whose divisor consists of a simple zero exactly on the orbit of the diagonal $\{(\tau,\tau): \, \tau \in \mathbb{H}\}$. From the formula of Theorem 1.2 we find an input form in weight $5$ for the Gram matrix $\mathbf{S} = \begin{psmallmatrix} 2 & 1 \\ 1 & -2 \end{psmallmatrix}$: $$R_{5,1/5,(2/5,1/5)}(\tau) = (q^{1/5} + 42q^{6/5} \pm ...)(\mathfrak{e}_{(2/5,1/5)} - \mathfrak{e}_{3/5,4/5}) + (-26q^{4/5} - 39q^{9/5} \pm ...) (\mathfrak{e}_{(1/5,3/5)} - \mathfrak{e}_{(4/5,2/5)}),$$ whose theta lift must equal $s_5$. As before, the fact that $s_5$ is obtained by lifting a form $R_{k,m,\beta}$ yields an identity, here $$s_5(\tau_1,\tau_2) = \frac{1}{120\pi i} \sum_{\substack{a,c \in \mathbb{Q} \\ b \in \mathbb{Q}(\sqrt{5}) \\ N_{K/\mathbb{Q}}(b) - ac = 1/5}} \mathrm{denom}\Big(a, b - \frac{1}{2} - \frac{1}{\sqrt{5}}, c \Big)^{-4} (a \tau_1 \tau_2 + b \tau_1 + b' \tau_2 + c)^{-5},$$ where $d = \mathrm{denom}(a,b,c)$ is minimal such that $da, dc \in \mathbb{Z}$ and $db \in \mathcal{O}_K$.
\end{ex}

\bibliographystyle{plainnat}
\bibliofont
\bibliography{\jobname}

\begin{thebibliography}{18}
\providecommand{\natexlab}[1]{#1}
\providecommand{\url}[1]{\texttt{#1}}
\expandafter\ifx\csname urlstyle\endcsname\relax
  \providecommand{\doi}[1]{doi: #1}\else
  \providecommand{\doi}{doi: \begingroup \urlstyle{rm}\Url}\fi

\bibitem[Borcherds(1998)]{Bn}
Richard Borcherds.
\newblock Automorphic forms with singularities on {G}rassmannians.
\newblock \emph{Invent. Math.}, 132\penalty0 (3):\penalty0 491--562, 1998.
\newblock ISSN 0020-9910.
\newblock \doi{10.1007/s002220050232}.
\newblock URL \url{http://dx.doi.org/10.1007/s002220050232}.

\bibitem[Bruinier and Bundschuh(2003)]{BB}
Jan Bruinier and Michael Bundschuh.
\newblock On {B}orcherds products associated with lattices of prime
  discriminant.
\newblock \emph{Ramanujan J.}, 7\penalty0 (1-3):\penalty0 49--61, 2003.
\newblock ISSN 1382-4090.
\newblock \doi{10.1023/A:1026222507219}.
\newblock URL \url{http://dx.doi.org/10.1023/A:1026222507219}.
\newblock Rankin memorial issues.

\bibitem[Bruinier and Kuss(2001)]{BK}
Jan Bruinier and Michael Kuss.
\newblock Eisenstein series attached to lattices and modular forms on
  orthogonal groups.
\newblock \emph{Manuscripta Math.}, 106\penalty0 (4):\penalty0 443--459, 2001.
\newblock ISSN 0025-2611.
\newblock \doi{10.1007/s229-001-8027-1}.
\newblock URL \url{http://dx.doi.org/10.1007/s229-001-8027-1}.

\bibitem[Cohen and Str\"{o}mberg(2017)]{CS}
Henri Cohen and Fredrik Str\"{o}mberg.
\newblock \emph{Modular forms (a classical approach)}, volume 179 of
  \emph{Graduate Studies in Mathematics}.
\newblock American Mathematical Society, Providence, RI, 2017.
\newblock ISBN 978-0-8218-4947-7.

\bibitem[Dittmann et~al.(2015)Dittmann, Hagemeier, and Schwagenscheidt]{DHS}
Moritz Dittmann, Heike Hagemeier, and Markus Schwagenscheidt.
\newblock Automorphic products of singular weight for simple lattices.
\newblock \emph{Mathematische Zeitschrift}, 279\penalty0 (1):\penalty0
  585--603, Feb 2015.
\newblock ISSN 1432-1823.
\newblock \doi{10.1007/s00209-014-1383-6}.
\newblock URL \url{https://doi.org/10.1007/s00209-014-1383-6}.

\bibitem[Eichler and Zagier(1985)]{EZ}
Martin Eichler and Don Zagier.
\newblock \emph{The theory of {J}acobi forms}, volume~55 of \emph{Progress in
  Mathematics}.
\newblock Birkh\"auser Boston, Inc., Boston, MA, 1985.
\newblock ISBN 0-8176-3180-1.
\newblock \doi{10.1007/978-1-4684-9162-3}.
\newblock URL \url{http://dx.doi.org/10.1007/978-1-4684-9162-3}.

\bibitem[Gundlach(1963)]{Gu}
Karl-Bernhard Gundlach.
\newblock Die {B}estimmung der {F}unktionen zur {H}ilbertschen {M}odulgruppe
  des {Z}ahlk\"orpers {$\mathbb{Q}(\sqrt{5})$}.
\newblock \emph{Math. Ann.}, 152:\penalty0 226--256, 1963.
\newblock ISSN 0025-5831.
\newblock \doi{10.1007/BF01470882}.
\newblock URL \url{https://doi.org/10.1007/BF01470882}.

\bibitem[McGraw(2003)]{Mc}
William McGraw.
\newblock The rationality of vector valued modular forms associated with the
  {W}eil representation.
\newblock \emph{Math. Ann.}, 326\penalty0 (1):\penalty0 105--122, 2003.
\newblock ISSN 0025-5831.
\newblock \doi{10.1007/s00208-003-0413-1}.
\newblock URL \url{https://doi.org/10.1007/s00208-003-0413-1}.

\bibitem[Milnor and Husemoller(1973)]{MH}
John Milnor and Dale Husemoller.
\newblock \emph{Symmetric bilinear forms}.
\newblock Springer-Verlag, New York-Heidelberg, 1973.
\newblock Ergebnisse der Mathematik und ihrer Grenzgebiete, Band 73.

\bibitem[Oda(1977/78)]{O}
Takayuki Oda.
\newblock On modular forms associated with indefinite quadratic forms of
  signature {$(2, n-2)$}.
\newblock \emph{Math. Ann.}, 231\penalty0 (2):\penalty0 97--144, 1977/78.
\newblock ISSN 0025-5831.
\newblock \doi{10.1007/BF01361138}.
\newblock URL \url{https://doi.org/10.1007/BF01361138}.

\bibitem[Raum(2016)]{R}
Martin Raum.
\newblock Computing genus 1 {J}acobi forms.
\newblock \emph{Math. Comp.}, 85\penalty0 (298):\penalty0 931--960, 2016.
\newblock ISSN 0025-5718.
\newblock \doi{10.1090/mcom/2992}.
\newblock URL \url{http://dx.doi.org/10.1090/mcom/2992}.

\bibitem[Schwagenscheidt(2018)]{Sch}
Markus Schwagenscheidt.
\newblock Eisenstein series for the {W}eil representation.
\newblock \emph{J. Number Theory}, 193:\penalty0 74--90, 2018.
\newblock ISSN 0022-314X.
\newblock \doi{10.1016/j.jnt.2018.05.014}.
\newblock URL \url{https://doi.org/10.1016/j.jnt.2018.05.014}.

\bibitem[Schwagenscheidt and Williams(2019)]{SW}
Markus Schwagenscheidt and Brandon Williams.
\newblock Twisted component sums of vector-valued modular forms.
\newblock Preprint, 2019.
\newblock URL \url{https://arxiv.org/abs/1903.07701}.

\bibitem[Williams(2018{\natexlab{a}})]{W}
Brandon Williams.
\newblock Poincar\'e square series for the {W}eil representation.
\newblock \emph{Ramanujan J.}, 2018{\natexlab{a}}.
\newblock \doi{10.1007/s11139-017-9986-2}.
\newblock URL \url{https://doi.org/10.1007/s11139-017-9986-2}.

\bibitem[Williams(2018{\natexlab{b}})]{W5}
Brandon Williams.
\newblock Vector-valued {H}irzebruch-{Z}agier series and class number sums.
\newblock \emph{Res. Math. Sci.}, 5\penalty0 (2):\penalty0 Paper No. 25, 13,
  2018{\natexlab{b}}.
\newblock ISSN 2522-0144.
\newblock \doi{10.1007/s40687-018-0142-4}.
\newblock URL \url{https://doi.org/10.1007/s40687-018-0142-4}.

\bibitem[Williams(2019{\natexlab{a}})]{W2}
Brandon Williams.
\newblock Poincar\'e square series of small weight.
\newblock \emph{Ramanujan J.}, 2019{\natexlab{a}}.
\newblock \doi{10.1007/s11139-018-0002-2}.
\newblock URL \url{https://doi.org/10.1007/s11139-018-0002-2}.

\bibitem[Williams(2019{\natexlab{b}})]{W3}
Brandon Williams.
\newblock Remarks on the theta decomposition of vector-valued {J}acobi forms.
\newblock \emph{J. Number Theory}, 197:\penalty0 250--267, 2019{\natexlab{b}}.
\newblock ISSN 0022-314X.
\newblock \doi{10.1016/j.jnt.2018.08.013}.
\newblock URL \url{https://doi.org/10.1016/j.jnt.2018.08.013}.

\bibitem[Williams(2019{\natexlab{c}})]{W4}
Brandon Williams.
\newblock Vector-valued {E}isenstein series of small weight.
\newblock \emph{Int. J. Number Theory}, 15\penalty0 (2):\penalty0 265--287,
  2019{\natexlab{c}}.
\newblock ISSN 1793-0421.
\newblock \doi{10.1142/S1793042119500118}.
\newblock URL \url{https://doi.org/10.1142/S1793042119500118}.

\end{thebibliography}

\end{document}